\documentclass[11pt]{article}
\usepackage{amsfonts}
\usepackage{mathrsfs,dsfont,bbm}
\usepackage[latin1]{inputenc}
\usepackage{amsmath,amssymb}
\usepackage{latexsym}
\usepackage{epstopdf}
\usepackage{geometry}
\usepackage[active]{srcltx}
\usepackage{graphicx}
\usepackage[
bookmarks=true,         
bookmarksnumbered=true, 
colorlinks=true, pdfstartview=FitV, linkcolor=blue, citecolor=blue,
urlcolor=blue]{hyperref}

\newtheorem{theorem}{Theorem}[section]

\newtheorem{lemma}[theorem]{Lemma}
\newtheorem{proposition}[theorem]{Proposition}

\numberwithin{equation}{section}
\newenvironment{proof}[1][Proof]{\noindent\textit{#1.} }{\hfill \rule{0.5em}{0.5em}}

\begin{document}
\date{\today}

\title{An estimation of Fisher information bound for distribution-dependent SDEs driven by fractional Brownian motion with small noise}

\author
{Tongxuan Liu, Qian Yu}

\maketitle

\begin{abstract}
\noindent In this paper, we consider the following distribution-dependent SDE driven by fractional Brownian motion with small noise
\begin{equation*}
    X_{t,\varepsilon}=X_0+\int_0^t{b( s,X_{s,\varepsilon},\mathcal{L}_{X_{s,\varepsilon}} )}ds
    +\varepsilon^H \int_0^t{\sigma ( s,\mathcal{L}_{X_{s,\varepsilon}} )}dB_{s}^{H}, t\in[0,T],
\end{equation*}
where the initial condition $X_0$ is a real number, $\mathcal{L}_{X_{t,\varepsilon}}$ denotes the law of $X_{t,\varepsilon}$, $\{B_{H}^{t}, t\geq0\}$ is a
fractional Brownian motion with Hurst index $H \in (\frac{1}{2},1)$ and
$\varepsilon $ is a small parameter.
We study the rate  of Fisher information convergence in the central limit theorem for the solution of small noise SDE, then we show that the convergence
rate $\varepsilon^{2H}$ is of optimal order.


\end{abstract}

\section{Introduction}{\label{Introduction}}

Given a random variable $F$ with an absolutely continuous density $p_F$,
the Fisher information of $F$ is defined by
\begin{equation*}
    I( F ) =\int_{-\infty}^{+\infty}\frac{[p'_{F}( x )] ^2}{p_F( x )}dx
    =\mathbb{E}[\rho_{F}^{2}(F)]
\end{equation*}
where $p'_{F}$ denotes a Radon-Nikodym derivative of $p_F$ and
$\rho_F=\frac{p'_F}{p_{F}}$ is the score function.
Let $N$ be a  normal distribution $\mathcal{N}(\mu,\sigma^2)$, then the Fisher information distance of $F$ to $N$ can be defined by
\begin{equation*}
    I(F\|N)=\mathbb{E}\left[\left ( \rho _F( F ) +\frac{F-\mu}{\sigma ^2} \right)^2 \right] .
\end{equation*}
Moreover, if the derivative $p'_F$ does not exist, the Fisher information distance
is defined to be infinite.

The research on Fisher's information convergence can be traced back to a paper published by Linnik \cite{ref14} in 1959. This has aroused the research interest of many scholars and produced many research results on Fisher information convergence. However, most of the existing results are devoted to the sums of independent random variables (see \cite{ref1}, \cite{ref2}, \cite{ref3},  \cite{ref5}, \cite{ref4} and references therein). Recently, Nourdin and Nualart \cite{ref6} used Malliavin
calculus method to obtain quantitative Fisher information bounds for the multiple Wiener-It{\^o} integrals. This provides a basis for Dung and Hang \cite{ref15} to establish the bounds of the Fisher information distance to the class of normal distributions of Malliavin differentiable random variables. Furthermore, they study the rate of Fisher information convergence in the central limit theorem for the solution of small noise
stochastic differential equations (SDEs):
$$
X_{\varepsilon,t}=X_0+\int_0^tb(s,X_{\varepsilon,s})ds+\varepsilon\int_0^t\sigma(s,X_{\varepsilon,s})dW_s, ~~t\in[0,T],
$$
where the initial condition $X_0=x_0\in \mathbb{R}$, $b, \sigma: [0,T]\times\mathbb{R}\to\mathbb{R}$ are deterministic functions, $W$ is a standard Brownian motion and $\varepsilon\in(0,1)$ is a small parameter.

The above SDEs consider the case of distribution independent and can be used to characterize linear Fokker-Planck equations.
Generally, nonlinear Fokker-Planck equations can be characterised by distribution dependent SDEs, which are also named as McKean-Vlasov SDEs or mean field SDEs.
A distinct feature of such systems is the appearance of probability laws in the coefficients of the resulting equations.
Wang \cite{Wang18} established strong well-posedness of distribution dependent SDEs with one-sided Lipschitz continuous drifts and Lipschitz-continuous dispersion coefficients.
Under integrability conditions on distribution dependent coefficients, Huang and Wang \cite{HW19} obtained the existence and uniqueness for distribution dependent SDEs with non-degenerate noise.
Recently, Fan \emph{et al.} \cite{Fan22} proved the well-posedness of distributed dependent SDE driven by fractional Brownian motion (fBm) and then established the Bismut formulas for both non-degenerate and degenerate cases.
Galeati \emph{et al.} \cite{GHM} studied distributed dependent SDEs with irregular, possibly distributional drift, driven by additive fBm and established strong well-posedness under a variety of assumptions on the drifts.

To our knowledge, there is a certain distance between distribution dependent SDE and distribution independent SDE, and also significant gap in the estimation of Fisher information bound for SDE in these two cases driven by fBm.
Note that, the study of Fisher information convergence in the central limit theorem
for the solution of  distribution-dependent SDE driven by fBm with small noise has not been studied. Thus, in this paper we will consider the  following distribution dependent SDE driven by fBm with small noise:
\begin{equation}{\label{SDE}}
    X_{t,\varepsilon}=X_0+\int_0^t{b( s,X_{s,\varepsilon},\mathcal{L}_{X_{s,\varepsilon}} )}ds
    +\varepsilon^H \int_0^t{\sigma ( s,\mathcal{L}_{X_{s,\varepsilon}} )}dB_{s}^{H}, ~~X_0=x, ~t\in[0,T],
\end{equation}
where the initial condition $X_0=x_0$ is a real number, $\mathcal{L}_{X_{t,\varepsilon}}$ denotes the law of $X_{t,\varepsilon}$, $B^{H}$ is a fBm with Hurst index $H \in (\frac{1}{2},1)$ and $\varepsilon\in(0,1) $ is a small parameter.

For some $\theta \in [1,\infty)$, let $\mathcal{F}_{\theta}(\mathbb{R})$ be the space of probability measures on $\mathbb{R}$ with finite $\theta$-th
moments. We define the $L^\theta$-Wasserstein distance of any two probability measures $\mu, \nu \in \mathcal{F}_{\theta}(\mathbb{R}) $ by
\begin{equation}{\label{df of W distance}}
    \mathds{W}_{\theta}( \mu,\nu ) :=\inf_{\pi \in \ell ( \mu,\nu )}\left( \int_{\mathbb{R}\times \mathbb{R}}{|x-y|^{\theta}\pi ( dxdy )} \right) ^{1/\theta},
\end{equation}
where $\ell(\mu,\nu)$ is the set of probability measures on $\mathbb{R} \times \mathbb{R}$
with marginals $\mu$ and $\nu$.
In order to ensure the existence and uniqueness of the solution, we impose
the following assumptions on the coefficients $b: [0,T]\times\mathbb{R}\times \mathcal{F}_{\theta}(\mathbb{R})\to\mathbb{R}$ and $\sigma: [0,T]\times\mathcal{F}_{\theta}(\mathbb{R})\to\mathbb{R}\otimes\mathbb{R}$.
Throughout this paper, we use $|\cdot|$ and $\langle\cdot \rangle$ for the Euclidean norm and inner product, respectively, and for a matrix, $||\cdot||$ denotes the operator norm.

Let $f: \mathcal{F}_{2}(\mathbb{R})\to \mathbb{R}$, $f$ is called $L$-differentiable at $\mu\in \mathcal{F}_{2}(\mathbb{R})$, if the functional $L^2(\mathbb{R}\to\mathbb{R}, \mu)\ni \phi\mapsto f(\mu\circ(Id+\phi)^{-1})$ is Fr\'{e}chet differentiable at $0\in L^2(\mathbb{R}\to\mathbb{R}, \mu)$.
A function $g$ is called differentiable on $\mathbb{R}\times \mathcal{F}_{2}(\mathbb{R})$, if for any $(x,\mu) \in \mathbb{R}\times \mathcal{F}_{2}(\mathbb{R})$, $g(\cdot,\mu)$
is differentiable and $g(x,\cdot)$ is $L$-differentiable. Furthermore, if $\nabla g(\cdot, \mu)(x)$ and
$D^Lg(x,\cdot)(\mu)(y)$ are jointly continuous in $(x,y,\mu)\in\mathbb{R}\times\mathbb{R}\times \mathcal{F}_{2}(\mathbb{R})$, we denote $g\in C^{1,(1,0)}(\mathbb{R}\times \mathcal{F}_{2}(\mathbb{R}))$.
More details of $L$-differentiable see in \cite{Fan22, Fan23}.

\begin{enumerate}
  \item[(A1)] There exists a non-decreasing function $K(t)$, such that for any
$t \in [0,T]$, $x,y \in \mathbb{R}$, $\mu, \nu \in \mathcal{F}_{\theta}(\mathbb{R}) $,
\begin{equation*}{\label{A1 1}}
    |b( t,x,\mu ) -b( t,y,\nu ) |\le K( t ) ( |x-y|+\mathds{W}_{\theta}( \mu,\nu ) ) ,
\end{equation*}
\begin{equation*}{\label{A1 2}}
    \lVert \sigma ( t,\mu ) -\sigma ( t,\nu ) \rVert \le K( t )\mathds{W}_\theta(\mu,\nu) ,
\end{equation*}
and
\begin{equation*}{\label{A1 3}}
    |b(t,0,\delta_0)|+ \lVert \sigma(t,\delta_0)\rVert \le K(t),
\end{equation*}
where $\delta_0$ denotes the initial experience distribution.

\item[(A2)] For every $t\in[0,T], b(t,\cdot,\cdot)\in C^{1,(1,0)}(\mathbb{R}\times \mathcal{F}_{2}(\mathbb{R}))$, and there exists a non-decreasing
function $K_1(t)$ such that

(i) for any $x,y\in\mathbb{R}, \mu,\nu\in\mathcal{F}_{2}(\mathbb{R})$,
$$|\nabla b(t,\cdot,\mu)(x)|+|D^L b(t,x,\cdot)(\mu)(y)|\leq K_1(t).$$

(ii) for any $x,y,z_1,z_2\in\mathbb{R}, \mu,\nu\in\mathcal{F}_{2}(\mathbb{R})$, $b(t,x,\mu)$ is twice differentiable in $x$ and
\begin{align*}
&|\nabla b(t,\cdot,\mu)(x)-\nabla b(t,\cdot,\nu)(y)|+|D^Lb(t,x,\cdot)(\mu)(z_1)-D^Lb(t,y,\cdot)(\nu)(z_2)|\\
&\qquad\leq K_1(t)(|x-y|+|z_1-z_2|+\mathds{W}_{\theta}( \mu,\nu )).
\end{align*}
\end{enumerate}

By Fan et al. \cite{Fan23}, under the assumption (A1), there exists a unique function $\{x_t, t\in[0,T]\}$ such that $x_t\in C([0,T];\mathbb{R})$ and $x_t$ satisfies the deterministic equation
\begin{equation}{\label{ODE}}
x_t=x_0+\int_0^t{b( s,x_{s},\mathcal{L}_{x_{s}} )}ds, ~~t\in[0,T].
\end{equation}
Under the assumption (A2), as $\varepsilon\to0$,
$\tilde{X}_{t,\varepsilon}=\frac{X_{t,\varepsilon}-x_t}{\varepsilon^H}
$ converges in distribution to $Z_t$, for every $t \in [0,T]$, where
$Z_t$ satisfies
\begin{align*}
Z_t&=\int_0^t\nabla_{Z_s}b(s,\cdot,\mathcal{L}_{x_s})(x_s)ds+\int_0^t\left(\mathbb{E}(D^Lb(s,u,\cdot)(\mathcal{L}_{x_s})(x_s)Z_s)|_{u=x_s}\right)ds\\
&\qquad+\int_0^t\sigma(s,\mathcal{L}_{x_s})dB_s^H, ~~t\in[0,T].
\end{align*}
Note that
$$\mathbb{E}\left[ D^Lb(s,u,\cdot)(\mathcal{L}_{x_s})(x_s) Z_s\right]=\lim_{a\to0}\frac{b(s,u,\mathcal{L}_{x_s+aZ_s})-b(s,u,\mathcal{L}_{x_s})}{a}.$$

This result provides a foundation for our study of the convergence in Fisher information distance.
In this paper, we mainly focus on the following two results.
The first one (Theorem \ref{Th1}) is the order of convergence in Fisher information distance,
$$
I\left( \tilde{X}_{t,\varepsilon}||Z_t \right) \le C_{upper}(t,H)\varepsilon^{2H};
$$
The second one (Theorem \ref{Th2}) is that the rate $O(\varepsilon^{2H})$ is optimal,
$$
\lim_{\varepsilon \rightarrow 0}\frac{1}{\varepsilon ^{2H}}I\left( \tilde{X}_{t,\varepsilon}||Z_t \right) \ge C_{lower}(t,H).
$$

The rest of this paper is orgnized as follows. In Section \ref{Perliminaries}, we present some preliminary properties of fBm, Malliavin calculus and some basic propositions.
In Section \ref{M R}, we impose our main results.
The rate of Fisher information convergence for the solutons of \eqref{SDE} is given in Theorem \ref{Th1} and
Theorem \ref{Th2} proves that this rate is optimal. Section \ref{Po1} and Section \ref{Po2} provide detailed proofs for
Theorem \ref{Th1} and Theorem \ref{Th2}, respectively.

\section{Perliminaries}{\label{Perliminaries}}
In this section, we present some preliminary properties and some basic propositions that will be used in
the rest of our paper.

\subsection{\label{sec:level2}Fractional Brownian motion}
Let $\{B^H_t, t\geq0\}$ be a fBm with Hurst parameter $H\in(0,1)$. It is a centered Gaussian process with covariance function
$$\mathbb{E}(B_t^HB_s^H)=\frac12(t^{2H}+s^{2H}-|t-s|^{2H}),$$
for all $t,s\geq0$. This process was first introduced by Kolmogorov and studied by Mandelbrot and Van Ness \cite{Man68} and it became an object of intense study due to
its self-similarity, long range dependence and Gaussianity.
Note that $B^{\frac{1}{2}}_t$ is a standard Brownian motion. To ensure the existence and uniqueness of the solution in equation \eqref{SDE}, we only consider fBm with Hurst index $H>1/2$.

FBm  admits the following Wiener integral representation
\begin{equation*}{\label{FBM int}}
    B_{t}^{H} = \int_0^t{K_H(t,s)dW_s},
\end{equation*}
where $W$ is a standard Brownian motion an $K_H$ is the kernel function defined by
\begin{equation*}
    K_H(t,s)=C_H s^{\frac{1}{2}-H} \int_s^t{(u-s)^{H-\frac{2}{3}}u^{H-\frac{1}{2}}du}, \ \ s \le t,
\end{equation*}
with $C_H=\left(\frac{H(2H-1)}{\beta(2-2H,H-\frac{1}{2})}   \right)^{\frac{1}{2}}$ and $\beta$ denotes a beta function.
Moverover, we can see that
\begin{equation*}
    \frac{\partial K_H}{\partial t}(t,s)=c_H\left(H-\frac{1}{2}\right)(t-s)^{H-\frac{3}{2}}\left(\frac{s}{t}\right)^{\frac{1}{2}-H}.
\end{equation*}

Next, we will introduce the Cameron-Martin space (more details see in \cite{ref21}). The Cameron-Martin space $\mathcal{K}_H$ associated to the covariance
$\mathbb{E}[B_{t}^{H}B_{s}^{H}]$ is defined as the closure of the space of step functions with respect to the scalar product
\begin{equation*}
    \langle \phi ,\psi \rangle_{\mathcal{K}_H}=c_H\int_0^T{\int_0^T{\phi(u)\psi(u)|u-v|^{2H-2}dudv}}.
\end{equation*}

For any $\phi \in \mathcal{K}_H$, we can define an operator $K^*_H$,
\begin{equation}{\label{fBm 1}}
    K^*_H(\phi)(s)=\int_s^T{\phi(t)\frac{\partial K_H}{\partial t}(t,s)dt}.
\end{equation}
Then $ (K^*_H1_{[0,t]})(s)=K_H(t,s)1_{[0,t]}(s)$ and  we have an isometry between $\mathcal{K}_H$ and $L^2[0,T]$,
\begin{equation*}
    \left<\phi , \psi \right>_{\mathcal{K}_H}=\langle K_H^*(\phi), K_H^*(\psi)\rangle_{L^2[0,T]}.
\end{equation*}
Thus, the Wiener integral with respect to fBm $B^{H}$ can be rewritten as a Wiener integral with respect to Wiener process $W$
\begin{equation}{\label{fBm 2}}
    \int_0^T \phi(s)dB^{H}_s= \int_0^T(K_H^*\phi)(s)dW_s.
\end{equation}

\subsection{\label{sec:level2}Malliavin calculus}
In this section, let us recall some elements of Malliavin calculus (for more details see \cite{ref19}).
We suppose that $\{W_t, t \in [0,T]\}$ is defined on a complete probability space $(\Omega, \mathcal{F}, \mathbb{F}, P)$, where $\mathbb{F}=(\mathcal{F}_t)_{t\in[0,T]}$
is a natural filtration generated by the Brownian motion $W$.

Let $\mathcal{S}$ denotes a dense subset of $L^2(\Omega, \mathcal{F}, P)$
that consists of smooth random variables of the form
\begin{equation}{\label{sf}}
	F=f(W(h_1),W(h_2),\ldots,W(h_n)),
\end{equation}
where $n\in \mathbb{N}$, $f \in C_b^{\infty}(\mathbb{R}^n)$ and
$h_1,\ldots,h_n \in L^2[0,T]$. If $F$ has the form (\ref{sf}), we define
its Malliavin derivative as the process $DF:=D_tF$, $t\in [0,T]$ given by
\begin{equation*}
	D_tF = \sum_{k=1}^n{\frac{\partial f}{\partial x_k}(W(h_1),\ldots, W(h_n))h_k(t)}
	.\end{equation*}
More generally, for each $k \ge 1$, we can define the iterated derivative
operator on a cylindrical random variable by setting
\begin{equation*}
	D_{t_1,\ldots,t_k}^kF=D_{t_1}\cdots D_{t_k}F=\sum_{i_i,...,i_{k=1}}^n\frac{\partial ^kf}{\partial x_{i_1}\cdots\partial x_{i_k}}( W( h_1 ) ,...,W( h_n ) ) h_{i_1}(t_1)\otimes \cdots \otimes h_{i_k}(t_k),
\end{equation*}
where $ h_{i_1}\otimes \cdots \otimes h_{i_k}$ denotes the $k$-fold tensor product of
$h_{i_1},\ldots,h_{i_k}$.

For any $1 \le p,k \le \infty $, we denote by $\mathbb{D}^{k,p}$ the closure
of $\mathcal{S}$ with respect to the norm
\begin{equation*}
	\lVert F \rVert _{k,p}^{p}:=\mathbb{E}|F|^p+\sum_{i=1}^{k}\mathbb{E}(||D^kF||^p_{(L^2[0,T])^{\otimes i}}).
\end{equation*}
Observe that, by computing the Malliavin derivative of $X_{t,\varepsilon}$, combining \eqref{fBm 1} and \eqref{fBm 2}, we can easily get:
\begin{equation}{\label{p1 eq2}}
	D_rX_{t,\varepsilon}=\int_r^t{{\nabla b( s,\cdot,\mathcal{L}_{X_{s,\varepsilon}})(X_{s,\varepsilon})} D_rX_{s,\varepsilon}  ds}+{ \int_r^t{\varepsilon^H \sigma ( s,\mathcal{L}_{X_{s,\varepsilon}}) \frac{\partial K_H( s,r )}{\partial s}ds}  }
\end{equation}
and
\begin{align}{\label{p1 eq3}}
	D_{\theta}D_rX_{t,\varepsilon}
	=\int_{r\lor \theta}^t\Big(\nabla b( s,\cdot,\mathcal{L}_{X_{s,\varepsilon}})(X_{s,\varepsilon}) D_{\theta}D_rX_{s,\varepsilon}
	 +\nabla^2 b( s,\cdot,\mathcal{L}_{X_{s,\varepsilon}})(X_{s,\varepsilon})  D_rX_{s,\varepsilon}D_{\theta}X_{s,\varepsilon}\Big)   ds.
\end{align}

For any $F\in \mathbb{D}^{1,2}$, the Clark-Ocone formula says that
\begin{equation*}
	F-\mathbb{E}[F]=\int_0^T{\mathbb{E}[D_sF|\mathcal{F}_s]dW_s}.
\end{equation*}

An important operator in the Malliavin calculus theory is
the divergence operator $\delta$. It is defined as the adjoint of the
derivative operator $D$ in the following manner. The domain of
$\delta$ denoted by $Dom\delta$ which is the set of all functions
$u \in L^2(\Omega \times [0,T])$ such that
\begin{equation*}
	|\mathbb{E}(\left< DF,u  \right>_{L^2[0,T]})|\le C(u)||F||_{L^2(\Omega)},
\end{equation*}
where $C(u)$ is some positive constant depending on $u$.
In particular, if $u\in Dom\delta$, then $\delta(u)$ is characterized
by following relationships
\begin{equation}\label{dual 1}
	\delta(uF)=F\delta(u)-\left< DF,u  \right>_{L^2[0,T]},
\end{equation}
and
\begin{equation*}{\label{dual 2}}
	\mathbb{E}[\left< DF,u  \right>_{L^2[0,T]}]=\mathbb{E}[F\delta(u)], \ \ \ \ \forall F \in \mathbb{D}^{1,2}.
\end{equation*}

\subsection{\label{sec:level2}Some propositions}
In this section, we will introduce some propositions that will be used in the proof of our main results.

\begin{proposition}\label{Lemma 2}(\cite{Fan23}, Lemma 4.1) Suppose that $\sigma$ satisfies (A1) and $\mu \in C([0,T]; \mathcal{F}_p(\mathbb{R}^d))$ with $p \ge \theta$
and $p > \frac{1}{H}$. Then there is a constant $C \ge 0$ such that
\begin{equation}{\label{Lemma2 eq1}}
	\mathbb{E}\left(\sup_{t\in [0,T]}\left|\int_0^t \sigma(s,\mu_s)dB^H_s\right|^p\right)
	\le C \int_0^T||\sigma(s,\mu_s)||^pds.
\end{equation}
\end{proposition}

The next proposition is a general bound of Fisher information which is very important to
the proofs of our main results.
\begin{proposition}\label{Lemma1}(\cite{ref15}, Theorem 3.1) Let $F \in \mathbb{D}^{2,4}$ and N be a normal
	random variable and $N \sim \mathcal{N}(\mu,\sigma^2)$. Define
	\begin{equation*}
		\Theta :=\left< DF,u \right> _{L^2[ 0,T ]},
	\end{equation*}
	where $u_t:=\mathbb{E}[D_tF|\mathcal{F}_t]$, $t \in [0,T]$. Assume that
	$\Theta \ne 0$ $a.s.$ Then, we have
	\begin{equation*}\label{Th1 eq1}
		I( F||N ) \le c\left( \frac{1}{\sigma ^4}\left( \mathbb{E}[ F ] -\mu \right) ^2+A_F|Var\left( F \right) -\sigma ^2|^2+C_F\left( \mathbb{E}\lVert D\Theta \rVert _{L^2[ 0,T ]}^{4} \right) \right) ,
	\end{equation*}
	where $c$ is an absolute constant and $A_F$, $C_F$ are positive constants
	given by
	\begin{equation*}
		A_F:=\frac{1}{\sigma ^4}\left( \mathbb{E}\lVert u \rVert _{L^2[ 0,T ]}^{8}\mathbb{E}|\Theta |^{-8} \right) ^{\dfrac{1}{4}},
	\end{equation*}
	\begin{equation*}
		C_F:=A_F+\left( \mathbb{E}\lVert u \rVert _{L^2[ 0,T ]}^{8}\mathbb{E}|\Theta |^{-16} \right) ^{\dfrac{1}{4}}.
	\end{equation*}
\end{proposition}

Throughout this papar, $C$ denotes some positive constant which may
change from line to line.

\section{Main results}{\label{M R}}

\begin{theorem}{\label{Th1}}Let $\{X_{t,\varepsilon}, t \in [0,T]\}$ and $\{x_t, t \in [0,T]\}$	be the solutions to the equations \eqref{SDE} and \eqref{ODE}, respectively.
	Define
	\begin{equation*}
		\tilde{X}_{t,\varepsilon}=\frac{X_{t,\varepsilon}-x_t}{\varepsilon^H}
		,\ \ t \in [0,T].
	\end{equation*}
	Suppose the assumptions (A1) and (A2) hold and that for some $p_0>16$,
	\begin{equation}{\label{NonDec}}
		\left| \int_0^t{\left( \int_r^t{\sigma ( s,\mathcal{L}_{X_{s,\varepsilon}}) \frac{\partial K_H( s,r )}{\partial s}ds} \right) ^2dr} \right|^{-p_0}<\infty,\ \forall \varepsilon\in (0,1),\ t\in(0,T].
	\end{equation}
	Then, for all $\varepsilon$ is a small number
	and $t \in (0,T]$, we have
	\begin{align*}
		I( \tilde{X}_{t,\varepsilon}||Z_t )&\le C(t,H)\left( \frac{1}{\text{Var}( Z_t )^2}+\frac{1}{\text{Var}( Z_t )^2}\left( \left| \int_0^t{\left( \int_r^t{\sigma ( s,\mathcal{L}_{X_{s,\varepsilon}}) \frac{\partial K_H( s,r )}{\partial s}ds} \right) ^2dr} \right|^{-p_0} \right) ^{\frac{2}{p_0}}\right.\\
		&\ \ \ \ \ \ \ \ \  \ \ \ \ \ \left.+\left( \left| \int_0^t{\left( \int_r^t{\sigma ( s,\mathcal{L}_{X_{s,\varepsilon}}) \frac{\partial K_H( s,r )}{\partial s}ds} \right) ^2dr} \right|^{-p_0} \right) ^{\frac{4}{p_0}} \right) \varepsilon ^{2H}
		.
	\end{align*}
	where $C(t,H)$ is a positive constant only
	depending on $t$ and $H$, $Z_t$ satisfies
	\begin{align}{\label{Zt}}
		Z_t&=\int_0^t\nabla_{Z_s}b(s,\cdot,\mathcal{L}_{x_s})(x_s)ds+\int_0^t\left(\mathbb{E}(D^Lb(s,u,\cdot)(\mathcal{L}_{x_s})(x_s)Z_s)|_{u=x_s}\right)ds \notag\\
		&\qquad+\int_0^t\sigma(s,\mathcal{L}_{x_s})dB_s^H, ~~t\in[0,T].
		\end{align}
\end{theorem}

Note that, if $\sigma=1$, the condition \eqref{NonDec} degenerates to $\Big(\int_0^tK_H^2(t,r)dr\Big)^{-p_0}=t^{-2p_0H}<\infty$.  Therefore, condition \eqref{NonDec} is feasible.

Theorem \ref{Th1} tells us, the rate of Fisher information convergence in the centeral limit theorem for the solution of \eqref{SDE} is $O(\varepsilon^2)$.
Moreover, the next theorem will prove that this rate is optimal.

\begin{theorem}{\label{Th2}}Under the assumptions (A1) and (A2), for each $t\in(0,T]$, we have
	\begin{equation}{\label{Th2 eq1}}
		\lim_{\varepsilon \rightarrow 0}\frac{1}{\varepsilon ^{2H}}I( \tilde{X}_{t,\varepsilon}||Z_t ) \ge \frac{1}{4||DU_t||_{L^2[0,T]}^4}\left( \mathbb{E}|\mathbb{E}\left[ \delta ( V_tDU_t) |U_t \right] _t| \right) ^2.
	\end{equation}
	where $Z_t$ is defined as in Theorem \ref{Th1},
	$(U_t)_{t\in[0,T]}$ and $(V_t)_{t\in[0,T]}$ are stochastic processes defined by
	\begin{equation}{\label{U_t}}
		U_t=\int_0^t{{\nabla b( s,\cdot,\mathcal{L}_{X_{s,\varepsilon}})(x_s) U_sds}}-\frac{1}{\varepsilon^H}\int_0^t{( b( s,x_s,\mathcal{L}_{X_{s,\varepsilon}} ) -b( s,x_s,\mathcal{L}_{x_s} ) ) ds}+\int_0^t{\sigma ( s,\mathcal{L}_{x_s}) dB_{s}^{H}},
	\end{equation}
	and
	\begin{align}{\label{V_t}}
		&V_t=\int_0^t{\frac{1}{2}{{\nabla^2 b( s,\cdot,\mathcal{L}_{X_{s,\varepsilon}})(x_s)}}U_{s}^{2}
			+{\nabla b( s,\cdot,\mathcal{L}_{X_{s,\varepsilon}})(x_s) V_s } }ds \notag \\
		&\ \ \ \ \ \ \ \ \ \ +\frac{1}{\varepsilon^H}\int_0^t{( \sigma ( s,\mathcal{L}_{X_{s,\varepsilon}}) -\sigma ( s,\mathcal{L}_{x_s}) ) dB_{s}^{H}}.
	\end{align}
\end{theorem}

Note that, as $\varepsilon\to0$, there exist some $\mu,\nu\in\mathcal{F}_{2}(\mathbb{R})$ and $c_1, c_2$, such that
$$\varepsilon^{-H}( b( s,x_s,\mathcal{L}_{X_{s,\varepsilon}} ) -b( s,x_s,\mathcal{L}_{x_s} ))\to c_1 D^Lb(s,x,\cdot)(\mu)$$
and
$$\varepsilon^{-H}( \sigma ( s,\mathcal{L}_{X_{s,\varepsilon}}) -\sigma ( s,\mathcal{L}_{x_s}) )\to c_2 D^L\sigma(s,\cdot)(\nu),$$
since \eqref{p3 eq1} and assumption (A1).
This ensures the rationality of the definitions of $U$ and $V$.

\section{Proof of Theorem \ref{Th1}} {\label{Po1}}
By the Proposition \ref{Lemma1}, we can see that the estimate of Fisher information distance between
$\tilde{X}_{t,\varepsilon}$ and $Z_t$ can be divided into three parts, that is the estimation of $\left( \mathbb{E}[ F ] -\mu \right) ^2$, $Var\left( F \right) -\sigma ^2$ and $\mathbb{E}\lVert D\Theta \rVert _{L^2[ 0,T ]}^{4}$. So that, our proof of Theorem \ref{Th1} can be divided into three estimates for each part. Here are some lemmas we need in our proof.

\begin{lemma} {\label{p1}}Under the assumption (A1), the equation
	\eqref{SDE} has a unique solution $\{X_{t,\varepsilon}, t \in [0,T]\}$
	satisfying for each $p \ge \theta$ and $p > \frac{1}{H}$,
	\begin{equation}{\label{p1 eq1}}
		\sup_{t\in [ 0,T ]}\mathbb{E}|X_{t,\varepsilon}|^p\le C_1(t,p,H),
	\end{equation}
	where $C_1(t,p,H)$ is a positive constant only depending on $t$, $p$ and $H$.
	
\end{lemma}
\begin{proof}	By Theorem 3.1 in \cite{Fan22}, under the assumption (A1), the equation 	\eqref{SDE} has a unique solution $\{X_{t,\varepsilon}, t \in [0,T]\}$ with the form
	\begin{equation*}
		X_{t,\varepsilon}=X_0+\int_0^t{b( s,X_{s,\varepsilon},\mathcal{L}_{X_{s,\varepsilon}})}ds+\varepsilon^H \int_0^t{\sigma ( s,\mathcal{L}_{X_{s,\varepsilon}})}dB_{s}^{H} .
	\end{equation*}
	Then, by \eqref{Lemma2 eq1} and the H{\"o}lder inequality, we have
	\begin{align*}
		\mathbb{E}|X_{t,\varepsilon}|^p
		&\le C+C{\mathbb{E}\left| \int_0^t{b( s,X_{s,\varepsilon},\mathcal{L}_{X_{s,\varepsilon}})}ds \right|}^p+C\mathbb{E}\left| \int_0^t{\sigma ( s,\mathcal{L}_{X_{s,\varepsilon}})}dB_{s}^{H} \right|^p\\
		&\le C+C\int_0^t{\mathbb{E}|b( s,X_{s,\varepsilon},\mathcal{L}_{X_{s,\varepsilon}}) |^p}ds+C\int_0^t{\lVert \sigma ( s,\mathcal{L}_{X_{s,\varepsilon}}) \rVert ^p}ds\\
		&\le C+C\int_0^t{\mathbb{E}\left( 1+|X_{s,\varepsilon}|^p+|\mathds{W}_{\theta}( \mathcal{L}_{X_{s,\varepsilon}},\delta _0 ) |^p \right) ds}+C\int_0^t\Big(1+|\mathds{W}_{\theta}( \mathcal{L}_{X_{s,\varepsilon}},\delta _0 ) |^p\Big)ds,
	\end{align*}
	where we use the assumption (A1) in the last inequality.
	By \eqref{df of W distance}, since $p \ge \theta$, we can use the H{\"o}lder inequality to obtain that
	\begin{equation*}
		|\mathds{W}_{\theta}( \mathcal{L}_{X_{s,\varepsilon}},\delta _0 ) |^p \le (\mathbb{E}|X_{s,\varepsilon}|^\theta)^{\frac{p}{\theta}} \le \mathbb{E}|X_{s,\varepsilon}|^p.
	\end{equation*}
	So that
	\begin{equation*}
		\mathbb{E}|X_{t,\varepsilon}|^p\le C+C\int_0^t{\mathbb{E}|X_{s,\varepsilon}|^p}ds.
	\end{equation*}
	Then, by Gronwall's inequality, we can obtain
	\begin{equation*}
		\mathbb{E}|X_{t,\varepsilon}|^p\le Ce^{ct}\le C.
	\end{equation*}
	This completes the proof.
\end{proof}

Note that, by the assumption (A1) and the H{\"o}lder inequality, we have
 \begin{align}{\label{boundness}}
       &\|\sigma(t,\mathcal{L}_{X_{t,\varepsilon}})\| \le \|\sigma(t,\mathcal{L}_{X_{t,\varepsilon}})-\sigma(t,\delta_0)\|+||\sigma(t,\delta_0)\| \notag \\
       & \ \ \ \le K(t)(1+\mathds{W}_\theta(\mathcal{L}_{X_{t,\varepsilon}},\delta_0))
	   \le K(t)(1+(\mathbb{E}|X_{t,\varepsilon}|^\theta)^{\frac{1}{\theta}}) .
 \end{align}
  Combine \eqref{boundness} and \eqref{p1 eq1}, $\sigma$
is a bounded function. In the rest of our paper, we put the boundness of $\sigma$ into assumption (A1).

\begin{lemma}{\label{p2}} Suppose the assumptions (A1) and (A2) hold.
	Then, for each $p \ge \theta$ and $p > \frac{1}{H}$, we have
	\begin{equation}{\label{p2 eq1}}
		\sup_{t\in [ 0,T ]}\mathbb{E}|D_rX_{t,\varepsilon}|^p\le C_2(t,p,H)\varepsilon ^{pH}
	\end{equation}
	and
	\begin{equation}{\label{p2 eq2}}
		\sup_{t\in [ 0,T ]}\mathbb{E}|D_{\theta}D_rX_{t,\varepsilon}|^p\le C_3(t,p,H)\varepsilon ^{2pH},
	\end{equation}
	for any $\varepsilon$ is a small number, where $C_2(t,p,H)$, $C_3(t,p,H)$ are positive constants only
	depending on $t$, $p$ and $H$.
	
\end{lemma}
\begin{proof}
	Since the assumption (A1) and
	the H{\"o}lder inequality, we have
	\begin{align*}
		\mathbb{E}|D_rX_{t,\varepsilon}|^p&\le C\mathbb{E}\left| \int_r^t{{\nabla b( s,\cdot,\mathcal{L}_{X_{s,\varepsilon}} )(X_{s,\varepsilon})} D_rX_{s,\varepsilon}  ds} \right|^p\\
		&\ \ \ \ +C\mathbb{E}\left| {\left( \int_r^t{\varepsilon^H \sigma ( s,\mathcal{L}_{X_{s,\varepsilon}}) \frac{\partial K_H( s,r )}{\partial s}ds} \right) } \right|^p\\
		&\le C\int_r^t{\mathbb{E}\left[ {\left| \nabla b( s,\cdot,\mathcal{L}_{X_{s,\varepsilon}} )(X_{s,\varepsilon}) \right|^p}|D_rX_{s,\varepsilon}|^p \right]}ds+C\varepsilon ^{pH}\mathbb{E}\left| \int_r^t{\frac{\partial K_H( s,r )}{\partial s}ds} \right|^p\\
		&\le C\int_r^t{\mathbb{E}|D_rX_{s,\varepsilon}|^pds}+C\varepsilon ^{pH}|K_H( t,r )| ^p\\
		&\le C\int_r^t{\mathbb{E}|D_rX_{s,\varepsilon}|^pds}+C\varepsilon ^{pH},
	\end{align*}
	Then by Gronwall's inequality, we can obtain \eqref{p2 eq1}.
	It remains to prove \eqref{p2 eq2}. Recall the form of $D_\theta D_rX_{t.\varepsilon}$, that is \eqref{p1 eq3}, we have
	\begin{align*}
		\mathbb{E}|D_{\theta}D_rX_{t,\varepsilon}|^p
		&\le C\mathbb{E}\left| \int_{r\lor \theta}^t{{\nabla b( s,\cdot,\mathcal{L}_{X_{s,\varepsilon}} )(X_{s,\varepsilon})  D_{\theta}D_rX_{s,\varepsilon}  }}ds \right|^p\\
		&\qquad+C\mathbb{E}\left| \int_{r\lor \theta}^t{{{\nabla^2 b( s,\cdot,\mathcal{L}_{X_{s,\varepsilon}} )(X_{s,\varepsilon})  D_rX_{s,\varepsilon}D_{\theta}X_{s,\varepsilon}  }}ds} \right|^p.
	\end{align*}
	Then by the assumptions (A1)-(A2) and the H{\"o}lder inequality,
	\begin{align*}
		\mathbb{E}|D_{\theta}D_rX_{t,\varepsilon}|^p
		&\le C\int_{r\lor \theta}^t{\mathbb{E}\big[{{\left| \nabla b( s,\cdot,\mathcal{L}_{X_{s,\varepsilon}} )(X_{s,\varepsilon}) \right|^p}}|D_{\theta}D_rX_{s,\varepsilon}|^p \big] ds}\\
		&\qquad+C\int_{r\lor \theta}^t{\mathbb{E}\big[ {{\left| \nabla^2 b( s,\cdot,\mathcal{L}_{X_{s,\varepsilon}} )(X_{s,\varepsilon}) \right|^p}}|D_rX_{s,\varepsilon}D_{\theta}X_{s,\varepsilon}|^p \big] ds}\\
		&\le C\int_{r\lor \theta}^t{\mathbb{E}|D_{\theta}D_rX_{s,\varepsilon}|^p}ds+C\int_{r\lor \theta}^t{\mathbb{E}\big[ |D_rX_{s,\varepsilon}|^p|D_{\theta}X_{s,\varepsilon}|^p \big]}ds\\
		&\le C\int_{r\lor \theta}^t{\mathbb{E}|D_{\theta}D_rX_{s,\varepsilon}|^p}ds+C\int_{r\lor \theta}^t{\sqrt{\mathbb{E}|D_rX_{s,\varepsilon}|^{2p}\mathbb{E}|D_{\theta}X_{s,\varepsilon}|^{2p}}}ds.
	\end{align*}
	Consequently, we can use the estimate \eqref{p2 eq1} to get
	\begin{equation*}
		\mathbb{E}|D_{\theta}D_rX_{t,\varepsilon}|^p\le C\int_{r\lor \theta}^t{\mathbb{E}|D_{\theta}D_rX_{s,\varepsilon}|^p}ds+C\varepsilon ^{2pH}.
	\end{equation*}
	Thus, \eqref{p2 eq2} is obtained by Gronwall's inequality.
\end{proof}

The following lemma characterizes the distance relationship between $X_{t,\varepsilon}$ and $x_t$, which has been proven in \cite{Fan23}, so we directly present this result here.
\begin{lemma}{\label{p3}}(\cite{Fan23}, Lemma 4.2) Under the assumptions (A1), for each $p \ge \theta$ and $p > \frac{1}{H}$, we have
	\begin{equation}{\label{p3 eq1}}
		\mathbb{E}\left(\sup_{t\in [ 0,T ]}|X_{t,\varepsilon}-x_t|^{p}\right)\le C_4(t,p,H)\varepsilon ^{pH}\left(1+\sup_{t\in [ 0,T ]}|x_t|^p\right),
	\end{equation}
	for $\varepsilon$ is a small number and $t \in [0,T]$,
	where $C_4(t,p,H)$ is a positive constant
	depending on $t$, $p$ and $H$.
\end{lemma}

\begin{lemma}{\label{p4}}
	Let $\tilde{X}_{t,\varepsilon}=\frac{X_{t,\varepsilon}-x_t}{\varepsilon^H}$. Suppose the assumptions (A1) and (A2) hold, then we have
	\begin{equation}{\label{p4 eq1}}
		\left| \mathbb{E}[ \tilde{X}_{t,\varepsilon} ] -\mathbb{E}[ Z_t ] \right|^2\le C_5( t,H ) \varepsilon ^{2H} ,
	\end{equation}
	and
	\begin{equation}{\label{p4 eq2}}
		\left| Var\left( \tilde{X}_{t,\varepsilon} \right) -\text{Var}( Z_t ) \right|\le C_6(t,H)\varepsilon^H ,
	\end{equation}
	for $\varepsilon$ is a small number and $t \in [0,T]$,
	where $Z_t$ is defined as that in Theorem \ref{Th1}, $C_5(t,H)$ and $C_6(t,H)$ are positive constants only
	depending on $t$ and $H$.
\end{lemma}
\begin{proof}
	We first verify the estimate \eqref{p4 eq1}. By \eqref{SDE} and \eqref{Zt} together with the H{\"o}lder inequality,
	\begin{align*}
&\mathbb{E}|\tilde{X}_{t,\varepsilon}-Z_t|^p\\
&=\mathbb{E}\left| \int_0^t{\left(\frac{1}{\varepsilon ^H}\left( b( s,X_{s,\varepsilon},\mathcal{L}_{X_{s,\varepsilon}}) -b( s,x_s,\mathcal{L}_{x_s} ) \right) -\nabla _{Z_s}b( s,\cdot ,\mathcal{L}_{x_s} ) ( x_s ) \right)ds} \right. \\
&\qquad\qquad-\left.\int_0^t{\mathbb{E}\left( D^Lb\left( s,u,\cdot \right) \left( \mathcal{L}_{x_s} \right) ( x_s ) Z_s \right) |_{u=x_s}ds+\int_0^t{\big(\sigma ( s,\mathcal{L}_{X_{s,\varepsilon}} ) -\sigma ( s,\mathcal{L}_{x_s} )\big) dB_{s}^{H}}} \right|^p\\
&\le C\int_0^t{\mathbb{E}}\left| \frac{1}{\varepsilon ^H}\left( b( s,X_{s,\varepsilon},\mathcal{L}_{X_{s,\varepsilon}}) -b( s,x_s,\mathcal{L}_{x_s} ) \right) -\nabla _{\tilde{X}_{s,\varepsilon}}b\left( s,\cdot ,\mathcal{L}_{X_{s,\varepsilon}} \right) ( x_s ) \right.\\
&\qquad\qquad\qquad\qquad\qquad\qquad\qquad\qquad \qquad\qquad\qquad\left.-\mathbb{E}\left( D^Lb\left( s,u,\cdot \right) \left( \mathcal{L}_{x_s} \right) ( x_s ) \tilde{X}_{s,\varepsilon} \right) |_{u=x_s} \right|^pds\\
&\ \ \ \ +C\int_0^t{\mathbb{E}}\left| \nabla _{\tilde{X}_{s,\varepsilon}}b\left( s,\cdot ,\mathcal{L}_{X_{s,\varepsilon}} \right) ( x_s ) -\nabla _{Z_s}b( s,\cdot ,\mathcal{L}_{x_s} ) ( x_s ) \right|^pds\\
&\ \ \ \ +C\int_0^t\mathbb{E}\left| D^Lb\left( s,u,\cdot \right) ( \mathcal{L}_{x_s} ) ( x_s ) \left( \tilde{X}_{s,\varepsilon}-Z_s \right)|_{u=x_s}  \right|^pds
+C\mathbb{E}\left| \int_0^t{\big(\sigma ( s,\mathcal{L}_{X_{s,\varepsilon}} ) -\sigma ( s,\mathcal{L}_{x_s} )\big) dB_{s}^{H}} \right|^p.
\end{align*}
Then by repeating the proof of Theorem 3.9 in \cite{Fan23}, we can get that
\begin{equation*}
		\mathbb{E}|\tilde{X}_{t,\varepsilon}-Z_t|^p\le C\int_0^t{\mathbb{E}|\tilde{X}_{s,\varepsilon}-Z_s|^pds}+C\int_0^t{\mathbb{E}|X_{s,\varepsilon}-x_s|^pds}.
\end{equation*}
Thus, by the Lemma \ref{p3}
\begin{equation*}
		\mathbb{E}|\tilde{X}_{t,\varepsilon}-Z_t|^p\le C\int_0^t{\mathbb{E}|\tilde{X}_{s,\varepsilon}-Z_s|^pds}+C\varepsilon ^{pH}.
\end{equation*}
Next, we can use Gronwall's inequality to obtain that
\begin{equation}{\label{p4 eq3}}
		\mathbb{E}|\tilde{X}_{t,\varepsilon}-Z_t|^p\le C\varepsilon ^{pH}.
\end{equation}
 We take $p=2$ in \eqref{p4 eq3}, we can get
\begin{equation*}
		\left| \mathbb{E}[ \tilde{X}_{t,\varepsilon} ] -\mathbb{E}[ Z_t ] \right|^2\le \mathbb{E}|\tilde{X}_{t,\varepsilon}-Z_t|^2\le C\varepsilon ^{2H}.
\end{equation*}

It remains to prove \eqref{p4 eq2}.
\begin{equation}{\label{G1+G2}}
		\left| \text{Var}\left( \tilde{X}_{t,\varepsilon} \right) -\text{Var}\left( Z_t \right) \right|\le \left| \mathbb{E}[ \tilde{X}_{t,\varepsilon}^2 ] -\mathbb{E}[ Z_t^2 ]  \right|+\left| \left( \mathbb{E}[ \tilde{X}_{t,\varepsilon} ] \right) ^2-\left( \mathbb{E}[ Z_t ] \right) ^2 \right|=:G_1+G_2.
\end{equation}
Let us consider the estimate of $G_1$,
\begin{equation*}
		G_1=\left| \mathbb{E}[ \tilde{X}_{t,\varepsilon} ^2] -\mathbb{E}[ Z_t ^2]  \right|\le \mathbb{E}|\tilde{X}_{t,\varepsilon}+Z_t||\tilde{X}_{t,\varepsilon}-Z_t|.
\end{equation*}
Then, by the Cauchy-Schwarz inequality, we have
\begin{equation*}
		\left| \mathbb{E}[ \tilde{X}_{t,\varepsilon} ^2] -\mathbb{E}[ Z_t ^2]  \right|\le \sqrt{\mathbb{E}|\tilde{X}_{t,\varepsilon}+Z_t|^2\mathbb{E}|\tilde{X}_{t,\varepsilon}-Z_t|^2}.
\end{equation*}
Recall the definitions of $\tilde{X}_{t,\varepsilon}$ and $Z_t$, then by the assumptions (A1)--(A2)
and \eqref{p3 eq1} (take $p=2$) combining the H{\"o}lder inequality, we have
\begin{align}{\label{G1 eq1}}
&\mathbb{E}|\tilde{X}_{t,\varepsilon}+Z_t|^2 \notag\\
&=\mathbb{E}\left| \int_0^t{\left(\frac{1}{\varepsilon ^H}\left( b( s,X_{s,\varepsilon},\mathcal{L}_{X_{s,\varepsilon}}) -b( s,x_s,\mathcal{L}_{x_s} ) \right) +\nabla _{Z_s}b( s,\cdot ,\mathcal{L}_{x_s} ) ( x_s )\right) ds} \right. \notag\\
&\ \ \ \ \ \ \ +\left. \int_0^t{\mathbb{E}\left( D^Lb\left( s,u,\cdot \right) \left( \mathcal{L}_{x_s} \right) ( x_s ) Z_s \right) ds+\int_0^t{\left(\sigma ( s,\mathcal{L}_{X_{s,\varepsilon}} ) +\sigma ( s,\mathcal{L}_{x_s} )\right) dB_{s}^{H}}} \right|^2 \notag\\
&\le C\mathbb{E}\left| \int_0^t{\frac{1}{\varepsilon ^H}\left( b( s,X_{s,\varepsilon},\mathcal{L}_{X_{s,\varepsilon}}) -b( s,x_s,\mathcal{L}_{x_s} ) \right) ds} \right|^2+C\mathbb{E}\left| \int_0^t{\left(\sigma ( s,\mathcal{L}_{X_{s,\varepsilon}} ) +\sigma ( s,\mathcal{L}_{x_s} )\right) dB_{s}^{H}} \right|^2+C \notag \\
&\le C\int_0^t{\frac{1}{\varepsilon ^{2H}}\mathbb{E}|b( s,X_{s,\varepsilon},\mathcal{L}_{X_{s,\varepsilon}}) -b( s,x_s,\mathcal{L}_{x_s} ) |^2ds}+C\int_0^t{||\sigma ( s,\mathcal{L}_{X_{s,\varepsilon}} ) +\sigma ( s,\mathcal{L}_{x_s} ) ||^2ds}+C \notag\\
&\le C\int_0^t{\frac{1}{\varepsilon ^{2H}}\left( \mathbb{E}|X_{s,\varepsilon}-x_s|^2+\mathbb{W}_{\theta}\left( \mathcal{L}_{X_{s,\varepsilon}},\mathcal{L}_{x_s} \right) ^2 \right) ds}+C\int_0^t{\left(||\sigma ( s,\mathcal{L}_{X_{s,\varepsilon}} ) ||^2+||\sigma ( s,\mathcal{L}_{x_s} ) ||^2\right)ds}+C \notag \\
&\le C\int_0^t{\frac{1}{\varepsilon ^{2H}} \mathbb{E}|X_{s,\varepsilon}-x_s|^2  ds}+C\le C.
\end{align}
Hence, by \eqref{p4 eq3} (take $p=2$) together with \eqref{G1 eq1},
\begin{equation}{\label{G_1}}
	G_1=\left| \mathbb{E}[ \tilde{X}_{t,\varepsilon}^2 ] -\mathbb{E}[ Z_t^2 ]  \right| \le C  \varepsilon^H
	.
\end{equation}

Next, we consider $G_2$. By \eqref{p4 eq3} and the H{\"o}lder inequality, we can get
    \begin{align}{\label{G_2}}
	G_2&=
\left| \left( \mathbb{E}[ \tilde{X}_{t,\varepsilon} ] \right) ^2-\left( \mathbb{E}[ Z_t ] \right) ^2 \right|=\left| \mathbb{E}[ \tilde{X}_{t,\varepsilon} ] -\mathbb{E}[ Z_t ] \right|\left| \mathbb{E}[ \tilde{X}_{t,\varepsilon} ] +\mathbb{E}[ Z_t ] \right| \notag\\
&\le \sqrt{\mathbb{E}\left| \tilde{X}_{t,\varepsilon}-Z_t \right|^2}\sqrt{\mathbb{E}\left| \tilde{X}_{t,\varepsilon}+Z_t \right|^2} \notag\\
&\le C  \varepsilon^H.
	\end{align}
Combining \eqref{G1+G2}, \eqref{G_1} and \eqref{G_2}, we can obtain \eqref{p4 eq2}. This completes the proof.
\end{proof}

\begin{lemma}{\label{p5}}
	Let $\tilde{X}_{t,\varepsilon}=\frac{X_{t,\varepsilon}-x_t}{\varepsilon^H}$.
	Define
	\begin{equation*}
		\varTheta _{\tilde{X}_{t,\varepsilon}}:=\int_0^t{D_r\tilde{X}_{t,\varepsilon}\mathbb{E}\left[ D_r\tilde{X}_{t,\varepsilon}|\mathcal{F}_r \right] dr}.
	\end{equation*}
	Then, under the assumptions (A1) and (A2) and the condition \eqref{NonDec}, we have
	\begin{equation}{\label{p5 eq1}}
		\mathbb{E}|\varTheta _{\tilde{X}_{t,\varepsilon}}|^{-p}\le C_7(t,p,H)\left( \left| \int_0^t{\left( \int_r^t{\sigma ( s,\mathcal{L}_{X_{s,\varepsilon}}) \frac{\partial K_H( s,r )}{\partial s}ds} \right) ^2dr} \right|^{-p_0} \right) ^{\frac{p}{p_0}},
	\end{equation}
	where $C_7(t,p,H)$ is a positive constant only
	depending on $t$, $p$ and $H$.
\end{lemma}
\begin{proof}
	First, let us recall the form of $D_rX_{t,\varepsilon}$.
	Solving the equation \eqref{p1 eq2},
	 we can obtain
	 \begin{equation}{\label{solution of DrX}}
	 	D_rX_{t,\varepsilon}= \varepsilon^H\int_r^t{ \sigma ( s,\mathcal{L}_{X_{s,\varepsilon}}) \frac{\partial K_H( s,r )}{\partial s}e^{-\int_s^t{\nabla b( u,\cdot,\mathcal{L}_{X_{u,\varepsilon}} )(X_{u,\varepsilon}) du}}ds} .
	 \end{equation}
	By \eqref{solution of DrX} and the assumption (A1),
	\begin{align*}
		\varTheta _{\tilde{X}_{t,\varepsilon}}&=\frac{1}{\varepsilon ^{2H}}\int_0^t{D_rX_{t,\varepsilon}\mathbb{E}\left[ D_rX_{t,\varepsilon}|\mathcal{F}_r \right] dr}\\
		&=\int_0^t{}{\left( \int_r^t{\sigma ( s,\mathcal{L}_{X_{s,\varepsilon}}) \frac{\partial K_H( s,r )}{\partial s}e^{-\int_s^t{\nabla b( u,\cdot,\mathcal{L}_{X_{u,\varepsilon}} )(X_{u,\varepsilon}) du}}ds} \right)}\\
		& \ \ \ \ \ \times \mathbb{E}\left[ {\left( \int_r^t{\sigma ( s,\mathcal{L}_{X_{s,\varepsilon}}) \frac{\partial K_H( s,r )}{\partial s}e^{-\int_s^t{\nabla b( u,\cdot,\mathcal{L}_{X_{u,\varepsilon}} )(X_{u,\varepsilon}) du}}ds} \right) |\mathcal{F}_r} \right] dr\\
		&\ge \int_0^t{{\left( \int_r^t{\sigma ( s,\mathcal{L}_{X_{s,\varepsilon}}) \frac{\partial K_H( s,r )}{\partial s}e^{-CT}ds} \right)}\mathbb{E}\left[ {\left( \int_r^t{\sigma ( s,\mathcal{L}_{X_{s,\varepsilon}}) \frac{\partial K_H( s,r )}{\partial s}e^{-CT}ds} \right) |\mathcal{F}_r} \right] dr}\\
		&=C\,\int_0^t{\left( \int_r^t{\sigma ( s,\mathcal{L}_{X_{s,\varepsilon}}) \frac{\partial K_H( s,r )}{\partial s}ds} \right) \mathbb{E}\left[ \int_r^t{\sigma ( s,\mathcal{L}_{X_{s,\varepsilon}}) \frac{\partial K_H( s,r )}{\partial s}ds|\mathcal{F}_r} \right] dr}\\
		&=C\,\int_0^t{\left( \int_r^t{\sigma ( s,\mathcal{L}_{X_{s,\varepsilon}}) \frac{\partial K_H( s,r )}{\partial s}ds} \right) ^2dr}.
	\end{align*}
	Thus, by the H{\"o}lder inequality
	\begin{align*}
		\mathbb{E}|\varTheta _{\tilde{X}_{t,\varepsilon}}|^{-p}&\le C\,\mathbb{E}\left| \int_0^t{\left( \int_r^t{\sigma ( s,\mathcal{L}_{X_{s,\varepsilon}}) \frac{\partial K_H( s,r )}{\partial s}ds} \right) ^2dr} \right|^{-p}\\
		&=C\left| \int_0^t{\left( \int_r^t{\sigma ( s,\mathcal{L}_{X_{s,\varepsilon}}) \frac{\partial K_H( s,r )}{\partial s}ds} \right) ^2dr} \right|^{-p}\\
		&\le C\left( \left| \int_0^t{\left( \int_r^t{\sigma ( s,\mathcal{L}_{X_{s,\varepsilon}}) \frac{\partial K_H( s,r )}{\partial s}ds} \right) ^2dr} \right|^{-p_0} \right) ^{\frac{p}{p_0}}.
	\end{align*}

\end{proof}

\begin{lemma}{\label{p6}}
	Let $\varTheta _{\tilde{X}_{t,\varepsilon}}$ be given in Lemma \ref{p5}.
	Under the assumptions (A1) and (A2), we have
	\begin{equation}{\label{p6 eq1}}
		\mathbb{E}\lVert D\varTheta _{\tilde{X}_{t,\varepsilon}} \rVert _{L^2[ 0,T ]}^{4}\le C_8(t,H)\varepsilon ^{4H},
	\end{equation}
	for $\varepsilon$ is a small number and $t \in [0,T]$,
	where $C_8(t,H)$ is a positive constant only
	depending on $t$ and $H$.
\end{lemma}
\begin{proof}
	Write the Malliavin derivative of $\varTheta _{\tilde{X}_{t,\varepsilon}}$:
	\begin{equation*}
		D\varTheta _{\tilde{X}_{t,\varepsilon}}=\frac{1}{\varepsilon ^{2H}}\int_0^t{\big( D_{\theta}D_rX_{t,\varepsilon}\mathbb{E}[ D_rX_{t,\varepsilon}|\mathcal{F}_r ] +D_rX_{t,\varepsilon}\mathbb{E}[ D_{\theta}D_rX_{t,\varepsilon}|\mathcal{F}_r ] \big) dr}.
	\end{equation*}
	Hence, we have
	\begin{equation*}
		\lVert D\varTheta _{\tilde{X}_{t,\varepsilon}} \rVert _{L^2[ 0,T ]}^{4}=\frac{1}{\varepsilon ^{8H}}\left( \int_0^T{\left( \int_0^t{\big( D_{\theta}D_rX_{t,\varepsilon}\mathbb{E}[ D_rX_{t,\varepsilon}|\mathcal{F}_r ] +D_rX_{t,\varepsilon}\mathbb{E}[ D_{\theta}D_rX_{t,\varepsilon}|\mathcal{F}_r ] \big) dr} \right)}^2d\theta \right) ^2.
	\end{equation*}
	By the H{\"o}lder inequality, we get
	\begin{align*}
		\mathbb{E}\lVert D\varTheta _{\tilde{X}_{t,\varepsilon}} \rVert _{L^2[ 0,T ]}^{4}&\le \frac{T}{\varepsilon ^{8H}}\int_0^T{\mathbb{E}\left| \int_0^t{( D_{\theta}D_rX_{t,\varepsilon}\mathbb{E}[ D_rX_{t,\varepsilon}|\mathcal{F}_r ] +D_rX_{t,\varepsilon}\mathbb{E}[ D_{\theta}D_rX_{t,\varepsilon}|\mathcal{F}_r ] ) dr} \right|^4d\theta}\\
		&\le \frac{C}{\varepsilon ^{8H}}\int_0^T{\int_0^t{\mathbb{E}\Big| D_{\theta}D_rX_{t,\varepsilon}\mathbb{E}\left[ D_rX_{t,\varepsilon}|\mathcal{F}_r \right] +D_rX_{t,\varepsilon}\mathbb{E}\left[ D_{\theta}D_rX_{t,\varepsilon}|\mathcal{F}_r \right] \Big|^4drd\theta}}\\
		&\le \frac{C}{\varepsilon ^{8H}}\int_0^T{\int_0^t{\left( \mathbb{E}\Big|D_{\theta}D_rX_{t,\varepsilon}\mathbb{E}\left[ D_rX_{t,\varepsilon}|\mathcal{F}_r \right] \Big|^4+\mathbb{E}\Big|D_rX_{t,\varepsilon}\mathbb{E}\left[ D_{\theta}D_rX_{t,\varepsilon}|\mathcal{F}_r \right] \Big|^4 \right) drd\theta}}\\
		&\le \frac{C}{\varepsilon ^{8H}}\int_0^T{\int_0^t{\sqrt{\mathbb{E}|D_{\theta}D_rX_{t,\varepsilon}|^8\mathbb{E}|D_rX_{t,\varepsilon}|^8}drd\theta}}.
	\end{align*}
	Thus, use \eqref{p2 eq1} and \eqref{p2 eq2}, we can complete this proof.
\end{proof}

\textbf{Proof of Theorem} \ref{Th1}:
Let $u= \mathbb{E}[D_r\tilde{X}_{t,\varepsilon}|\mathcal{F}_r]$ for $0 \le r \le T$. Then, by Proposition \ref{Lemma1} ($\tilde{X}_{t,\varepsilon} \in \mathbb{D}^{2,4}$ is implied by Lemma \ref{p2}),
\begin{align}{\label{4.18}}
	I( \tilde{X}_{t,\varepsilon}||Z_t ) &\le c\left( \frac{1}{\text{Var}( Z_t )^2}\left( \mathbb{E}[ \tilde{X}_{t,\varepsilon} ] -\mathbb{E}[Z_t]\right) ^2+A_{\tilde{X}_{t,\varepsilon}}|Var\left( \tilde{X}_{t,\varepsilon} \right) -\text{Var}( Z_t )|^2 \notag \right.\\
	&\ \ \ \ \left.+C_{\tilde{X}_{t,\varepsilon}}\left( \mathbb{E}\lVert D\Theta _{\tilde{X}_{t,\varepsilon}} \rVert _{L^2[ 0,T ]}^{4} \right) \right) ,
\end{align}
where $c$ is an absolute constant and
\begin{equation*}
	A_{\tilde{X}_{t,\varepsilon}}:=\frac{1}{\text{Var}( Z_t )^2}\left( \mathbb{E}\lVert u \rVert _{L^2[ 0,T ]}^{8}\mathbb{E}|\Theta _{\tilde{X}_{t,\varepsilon}}|^{-8} \right) ^{\dfrac{1}{4}}
	,
\end{equation*}
\begin{equation*}
	C_{\tilde{X}_{t,\varepsilon}}:=A_{\tilde{X}_{t,\varepsilon}}+\left( \mathbb{E}\lVert u \rVert _{L^2[ 0,T ]}^{8}\mathbb{E}|\Theta _{\tilde{X}_{t,\varepsilon}}|^{-16} \right) ^{\dfrac{1}{4}}
	.
\end{equation*}

Using H{\"o}lder inequality and the estimate \eqref{p2 eq1} we have
\begin{align*}
	\mathbb{E}\lVert u \rVert _{L^2[ 0,T ]}^{8}&\le \frac{C}{\varepsilon ^{8H}}\int_0^t{\mathbb{E}|\mathbb{E}\left[ D_rX_{t,\varepsilon}|\mathcal{F}_r \right] |^8dr}\\
	&\le \frac{C}{\varepsilon ^{8H}}\int_0^t{\mathbb{E}|D_rX_{t,\varepsilon}|^8dr}\\
	&\le C.
\end{align*}
Combining \eqref{p4 eq1}, \eqref{p4 eq2}, \eqref{p5 eq1}, \eqref{p6 eq1} and \eqref{4.18},
we can obtain that
\begin{align*}
	I( \tilde{X}_{t,\varepsilon}||Z_t )&\le C(t,H)\left( \frac{1}{\text{Var}( Z_t )^2}+\frac{1}{\text{Var}( Z_t )^2}\left( \left| \int_0^t{\left( \int_r^t{\sigma ( s,\mathcal{L}_{X_{s,\varepsilon}}) \frac{\partial K_H( s,r )}{\partial s}ds} \right) ^2dr} \right|^{-p_0} \right) ^{\frac{2}{p_0}}\right.\\
	&\ \ \ \ \ \ \ \ \  \ \ \ \ \ \left.+\left( \left| \int_0^t{\left( \int_r^t{\sigma ( s,\mathcal{L}_{X_{s,\varepsilon}}) \frac{\partial K_H( s,r )}{\partial s}ds} \right) ^2dr} \right|^{-p_0} \right) ^{\frac{4}{p_0}} \right) \varepsilon ^{2H}
	.
\end{align*}

\section{Proof of Theorem \ref{Th2}}{\label{Po2}}
It is well known that the Fisher information and total variation distance satisfy the following relation (see in \cite{ref22} and \cite{ref23})
\begin{equation}{\label{FI vs TV}}
	\sqrt{I( \tilde{X}_{t,\varepsilon}||Z_t )}\ge d_{TV}( \tilde{X}_{t,\varepsilon},Z_t ) :=\frac{1}{2}\sup_g\left| \mathbb{E}\left[ g\left( \tilde{X}_{t,\varepsilon} \right) \right] -\mathbb{E}[ g\left( Z_t \right) ] \right|,
\end{equation}
where the supremum is running over all measurable function $g$ bounded by one.
Thus our main task is to find a lower bound for $\lim_{\varepsilon \rightarrow 0}d_{TV}( \tilde{X}_{t,\varepsilon},Z_t )
$. We use the same idea as Section 4.3 in \cite{ref15}, so that our proof can be divided into several lemmas.

\begin{lemma}{\label{p7}} Suppose the assumptions (A1) and (A2) hold.
	Then, for each $p \ge \theta$ and $p > \frac{1}{H}$, we have
	\begin{equation}{\label{p7 eq1}}
		\sup_{t\in [ 0,T ]}\mathbb{E}| \tilde{X}_{t,\varepsilon}-U_t |^p\le C_9(t,p,H)\varepsilon ^{pH},
	\end{equation}
	\begin{equation}{\label{p7 eq2}}
		\sup_{t\in [ 0,T ]} \mathbb{E}|U_t|^p \le C_{10}(t,p,H)
	\end{equation}
and
\begin{equation}{\label{p7 eq2-3}}
\sup_{t\in [ 0,T ]} \mathbb{E}|V_t|^p \le C_{11}(t,p,H),
\end{equation}
for any $\varepsilon$ is a small number, where $C_9(t,p,H)$, $C_{10}(t,p,H), C_{11}(t,p,H)$ are positive constants only depending on $t$, $p$ and $H$,
$U_t, V_t$ are defined in \eqref{U_t} and \eqref{V_t}.
\end{lemma}
\begin{proof}
	Recall the definitions of $\tilde{X}_{t,\varepsilon}$ in Theorem \ref{Th1} and $U_t$ in \eqref{U_t}, we have
	\begin{align}{\label{p7 eq3}}
		\tilde{X}_{t,\varepsilon}-U_t&=
		\frac{1}{\varepsilon^H}\int_0^t{\left( b( s,X_{s,\varepsilon},\mathcal{L}_{X_{s,\varepsilon}}) -b( s,x_s,\mathcal{L}_{X_{s,\varepsilon}} ) \right)}ds \notag \\
		&\qquad+\int_0^t{\left( \sigma ( s,\mathcal{L}_{X_{s,\varepsilon}}) -\sigma ( s,\mathcal{L}_{x_s}) \right) dB_{s}^{H}}-\int_0^t{{\nabla b( s,\cdot,\mathcal{L}_{X_{s,\varepsilon}})(x_s) U_sds}}.
	\end{align}
	Then by the Taylor's expansion
	\begin{align}{\label{p7 eq4}}
		&b( s,X_{s,\varepsilon},\mathcal{L}_{X_{s,\varepsilon}}) -b( s,x_s,\mathcal{L}_{X_{s,\varepsilon}} )\notag \\
		&={\nabla b( s,\cdot,\mathcal{L}_{X_{s,\varepsilon}})(x_s) ( X_{s,\varepsilon}-x_s ) } +\frac{1}{2}{{\nabla^2 b( s,\cdot ,\mathcal{L}_{X_{s,\varepsilon}} )(x_s+\eta ( X_{s,\varepsilon}-x_s )) ( X_{s,\varepsilon}-x_s )^2}},
	\end{align}
we can see
	\begin{align*}
		\mathbb{E}|\tilde{X}_{t,\varepsilon}-U_t|^p
		&\le C\mathbb{E}\left| \int_0^t{{\nabla b( s,\cdot,\mathcal{L}_{X_{s,\varepsilon}})(x_s) \left( \tilde{X}_{s,\varepsilon}-U_s \right) }ds} \right|^p\\
		&\ \ \ \ \ \ \ \ +\frac{C}{\varepsilon ^{pH}}\mathbb{E}\left| \int_0^t{{{\nabla^2 b( s,\cdot ,\mathcal{L}_{X_{s,\varepsilon}} )(x_s+\eta ( X_{s,\varepsilon}-x_s ))  ( X_{s,\varepsilon}-x_s )^2}}ds} \right|^p\\
		&\ \ \ \ \ \ \ \ +C\mathbb{E}\left| \int_0^t{\left( \sigma ( s,\mathcal{L}_{X_{s,\varepsilon}}) -\sigma ( s,\mathcal{L}_{x_s}) \right) dB_{s}^{H}} \right|^p .
	\end{align*}
	Then by the H{\"o}lder inequality and \eqref{Lemma2 eq1}, we have
	\begin{align*}
		\mathbb{E}|\tilde{X}_{t,\varepsilon}-U_t|^p
		&\le C\int_0^t{\mathbb{E}\big[ {\left| \nabla b( s,\cdot,\mathcal{L}_{X_{s,\varepsilon}})(x_s) \right|^p|\tilde{X}_{s,\varepsilon}-U_s|^p} \big] ds}\\
		&\ \ \ \ \ \ \ \ +\frac{C}{\varepsilon ^{pH}}\int_0^t{\mathbb{E}\big[{{\left| \nabla^2 b( s,\cdot ,\mathcal{L}_{X_{s,\varepsilon}} )(x_s+\eta ( X_{s,\varepsilon}-x_s ))  \right|^p|X_{s,\varepsilon}-x_s|^{2p}}} \big]}\\
		&\ \ \ \ \ \ \ \ +C\int_0^t{||\sigma ( s,\mathcal{L}_{X_{s,\varepsilon}}) -\sigma ( s,\mathcal{L}_{x_s}) ||^pds}.
	\end{align*}
	Thus by assumptions (A1)--(A2), \eqref{p3 eq1} and \eqref{df of W distance} combining the H{\"o}lder inequality, for any $p \ge \theta$,
	\begin{align*}
		\mathbb{E}|\tilde{X}_{t,\varepsilon}-U_t|^p&\le C\int_0^t{\mathbb{E}[ |\tilde{X}_{s,\varepsilon}-U_s|^p ] ds}+\frac{C}{\varepsilon ^{pH}}\int_0^t{\mathbb{E}[ |X_{s,\varepsilon}-x_s|^{2p} ]}
		+C\int_0^t{|\mathds{W}_{\theta}( \mathcal{L}_{X_{s,\varepsilon}},\mathcal{L}_{x_s} )| ^pds}\\
		&\le C\int_0^t{\mathbb{E}\left[ |\tilde{X}_{s,\varepsilon}-U_s|^p \right] ds}+C\varepsilon ^{pH}.
	\end{align*}
	Finally, we can use Gronwall's inequality to obtain (\ref{p7 eq1}).
	
	It remains to prove \eqref{p7 eq2} and \eqref{p7 eq2-3}.
	\begin{align*}
		\mathbb{E}|U_t|^p&\le C\mathbb{E}\left| \int_0^t{{\nabla b( s,\cdot,\mathcal{L}_{X_{s,\varepsilon}})(x_s) U_sds}} \right|^p\\
		&\ \ \ \ +\frac{C}{\varepsilon ^{pH}}\mathbb{E}\left| \int_0^t{( b( s,x_s,\mathcal{L}_{X_{s,\varepsilon}} ) -b( s,x_s,\mathcal{L}_{x_s} ) ) ds} \right|^p+C\mathbb{E}\left| \int_0^t{\sigma ( s,\mathcal{L}_{x_s}) dB_{s}^{H}} \right|^p.
	\end{align*}
	By the assumptions (A1)--(A2) and \eqref{Lemma2 eq1}, we have
	\begin{align*}
		\mathbb{E}|U_t|^p
		&\le C\,\int_0^t{\mathbb{E}}\left| {\nabla b( s,\cdot,\mathcal{L}_{X_{s,\varepsilon}})(x_s) U_s} \right|^pds\\
		&\ \ \ \ +\frac{C}{\varepsilon ^{pH}}\int_0^t{\mathbb{E}| b( s,x_s,\mathcal{L}_{X_{s,\varepsilon}} ) -b( s,x_s,\mathcal{L}_{x_s} )|^pds}+C\int_0^{t}||\sigma ( s,\mathcal{L}_{x_s}) ||^pds\\
		&\le C\int_0^t{\mathbb{E}}| U_s|^pds+\frac{C}{\varepsilon ^{pH}}\int_0^t{\mathbb{E}| \mathds{W}_{\theta}( \mathcal{L}_{X_{s,\varepsilon}},\mathcal{L}_{x_s} ) |^pds}+C\int_0^t{||\sigma ( s,\mathcal{L}_{x_s}) ||^pds}.
	\end{align*}
	Then, by \eqref{df of W distance} and the H{\"o}lder inequality, for any $p \ge \theta$,
	\begin{align*}
		\mathbb{E}|U_t|^p
		&\le C\int_0^t{\mathbb{E}}| U_s|^pds+\frac{C}{\varepsilon ^{pH}}\int_0^t{\mathbb{E}| X_{s,\varepsilon}-x_s |^pds}+C\\
		&\le C\int_0^t{\mathbb{E}}| U_s|^pds+C.
	\end{align*}
	Finally, we can use Gronwall's inequality to obtain \eqref{p7 eq2}.

Note that, by using the similar method as the proof of \eqref{p7 eq2}, we can easily get that
for any $p\ge\theta$ and $p>\frac{1}{H}$, $\mathbb{E}|V_t|^p \le C(t,p,H)$.
\end{proof}

\begin{lemma}{\label{p8}}
	Suppose the assumptions (A1) and (A2) hold.
	Then, we have
	\begin{equation*}
		\sup_{t\in [ 0,T ]}\int_0^t{\mathbb{E}|D_r\tilde{X}_{t,\varepsilon}-D_rU_t|^2dr}\le C_{12}(t,H)\varepsilon ^{2H},
	\end{equation*}
where  $C_{12}(t,H)$ is a positive constant only depending on $t$, $U_t$ is defined as \eqref{U_t}.
\end{lemma}

\begin{proof}
	Compute the Malliavin derivative of $U_t$
	\begin{equation*}
		D_rU_t=\int_r^t{{\nabla b( s,\cdot,\mathcal{L}_{X_{s,\varepsilon}})(x_s) D_rU_sds}}+ \int_r^t{\sigma ( s,\mathcal{L}_{x_s}) \frac{\partial K_H( s,r )}{\partial s}ds}.
	\end{equation*}
	By solving this equation, we can get
	\begin{equation*}
		D_rU_t= \int_r^t{\sigma ( s,\mathcal{L}_{x_s}) \frac{\partial K_H( s,r )}{\partial s}e^{-\int_s^t{\nabla b( u,\cdot,\mathcal{L}_{X_{u,\varepsilon}})(x_u) du}}ds}.
	\end{equation*}
	Recall the solution of $D_rX_{t,\varepsilon}$ in \eqref{solution of DrX}, and the definition of $D_r\tilde{X}_{t,\varepsilon}$,
	\begin{equation*}
		D_r\tilde{X}_{t,\varepsilon}=\frac{1}{\varepsilon^H}D_rX_{t,\varepsilon}= \int_r^t{\sigma ( s,\mathcal{L}_{X_{s,\varepsilon}}) \frac{\partial K_H( s,r )}{\partial s}e^{-\int_s^t{\nabla b( u,\cdot,\mathcal{L}_{X_{u,\varepsilon}} )(X_{u,\varepsilon}) du}}ds}.
	\end{equation*}
	Then, by the H{\"o}lder inequality and the assumptions (A1)--(A2)
	\begin{align*}
		&\mathbb{E}|D_r\tilde{X}_{t,\varepsilon}-D_rU_t|^2\\
		&=\mathbb{E}\left| \int_r^t{}\frac{\partial K_H( s,r )}{\partial s}\left[ \left( \sigma ( s,\mathcal{L}_{X_{s,\varepsilon}}) -\sigma ( s,\mathcal{L}_{x_s}) \right) e^{-\int_s^t{\nabla b( u,\cdot,\mathcal{L}_{X_{u,\varepsilon}} )(X_{u,\varepsilon}) du}}\right.\right.\\
		&\ \ \ \ \ \ \ \ \ \ \ \ \left.\left.+\sigma ( s,\mathcal{L}_{x_s}) \left( e^{-\int_s^t{\nabla b( u,\cdot,\mathcal{L}_{X_{u,\varepsilon}} )(X_{u,\varepsilon}) du}}-e^{-\int_s^t{\nabla b( u,\cdot,\mathcal{L}_{X_{u,\varepsilon}})(x_s) du}} \right) \right] ds \right|^2\\
		&\le {C}\left\{ \mathbb{E}\left| \int_r^t{\frac{\partial K_H( s,r )}{\partial s}\left( \sigma ( s,\mathcal{L}_{X_{s,\varepsilon}}) -\sigma ( s,\mathcal{L}_{x_s}) \right) ds} \right|^2\right.\\
		&\ \ \ \ \ \ \ \ \ \ \ \ \left.+\mathbb{E}\left| \int_r^t{\frac{\partial K_H( s,r )}{\partial s}\left(e^{-\int_s^t{\nabla b( u,\cdot,\mathcal{L}_{X_{u,\varepsilon}} )(X_{u,\varepsilon}) du}}-e^{-\int_s^t{\nabla b( u,\cdot,\mathcal{L}_{X_{u,\varepsilon}})(x_u) du}}\right) ds} \right|^2 \right\} \\
		&\le C\mathbb{E}\left| \sup_s\left( \sigma ( s,\mathcal{L}_{X_{s,\varepsilon}} ) -\sigma ( s,\mathcal{L}_{x_s} ) \right) \int_r^t{\frac{\partial K_H( s,r )}{\partial s}ds} \right|^2+C{\mathbb{E}\left|\sup_u(X_{u,\varepsilon}-x_u)\int_r^t{\frac{\partial K_H( s,r )}{\partial s}ds}\right|^2ds}\\
	    &\le C\mathbb{E}\left[ \sup_s|X_{s,\varepsilon}-x_s|^p \right] +C{\mathbb{E}\left[\sup_u|X_{u,\varepsilon}-x_u|^2\right]}.
	\end{align*}
	Thus, by \eqref{p3 eq1},
	\begin{equation*}
		\mathbb{E}|D_r\tilde{X}_{t,\varepsilon}-D_rU_t|^2 \le C\varepsilon^{2H}.
	\end{equation*}
\end{proof}
\begin{lemma}{\label{p9}}Under the assumptions (A1) and (A2), for each $p \ge \theta$ and $p > \frac{1}{H}$,
	we have
	\begin{equation}{\label{p9 eq1}}
		\lim_{\varepsilon \rightarrow 0}\mathbb{E}\left| \frac{\tilde{X}_{t,\varepsilon}-U_t}{\varepsilon^H}-V_t \right|^p=0,
	\end{equation}
	where $U_t$ and $V_t$ are defined as \eqref{U_t} and \eqref{V_t}, respectively.
\end{lemma}
\begin{proof}
	By \eqref{p7 eq3} and \eqref{p7 eq4},
	\begin{align*}
		\frac{\tilde{X}_{t,\varepsilon}-U_t}{\varepsilon^H}
		&=\int_0^t{{\nabla b( s,\cdot,\mathcal{L}_{X_{s,\varepsilon}})(x_s) \left( \frac{\tilde{X}_{s,\varepsilon}-U_s}{\varepsilon^H} \right) ds}}
		+\frac12\int_0^t{{ \nabla^2 b( s,\cdot ,\mathcal{L}_{X_{s,\varepsilon}} )(x_s+\eta ( X_{s,\varepsilon}-x_s ))  \tilde{X}_{s,\varepsilon}^{2}ds}}\\
		&\qquad+\frac{1}{\varepsilon^H}\int_0^t{\left(\sigma(s,\mathcal{L}_{X_{s,\varepsilon}})-\sigma(s,\mathcal{L}_{x_s})\right)dB_s^H}.
	\end{align*}
	Then, recall the definition of $V_t$, we have
	\begin{align*}
		\mathbb{E}\left| \frac{\tilde{X}_{t,\varepsilon}-U_t}{\varepsilon^H}-V_t \right|^p&\le C\,\mathbb{E}\left| \int_0^t{{\nabla b( s,\cdot,\mathcal{L}_{X_{s,\varepsilon}})(x_s) \left( \frac{\tilde{X}_{s,\varepsilon}-U_s}{\varepsilon^H} -V_s\right) ds}} \right|^p\\
		&+C\,\mathbb{E}\left| \int_0^t{{\left( \nabla^2 b( s,\cdot ,\mathcal{L}_{X_{s,\varepsilon}} )(x_s+\eta ( X_{s,\varepsilon}-x_s ))  \tilde{X}_{s,\varepsilon}^{2}-\nabla^2 b( s,\cdot,\mathcal{L}_{X_{s,\varepsilon}})(x_s) U_{s}^{2} \right)ds }} \right|^p.
	\end{align*}
	By the H{\"o}lder inequality,
	\begin{align*}
\mathbb{E}\left| \frac{\tilde{X}_{t,\varepsilon}-U_t}{\varepsilon ^H}-V_t \right|^p&\le C\int_0^t{\mathbb{E}\left| \frac{\tilde{X}_{s,\varepsilon}-U_s}{\varepsilon ^H}-V_s \right|}^pds\notag\\
&\ \ +C\int_0^t{\mathbb{E}\left| \nabla ^2b\left( s,\cdot ,\mathcal{L}_{X_{s,\varepsilon}} \right) \left( x_s+\eta \left( X_{s,\varepsilon}-x_s \right) \right) \left( \tilde{X}_{s,\varepsilon}^{2}-U_{s}^{2} \right) \right|^p}ds \notag\\
&\ \ +C\int_0^t{\mathbb{E}\left| \left( \nabla ^2b\left( s,\cdot ,\mathcal{L}_{X_{s,\varepsilon}} \right) \left( x_s+\eta \left( X_{s,\varepsilon}-x_s \right) \right) -\nabla ^2b\left( s,\cdot ,\mathcal{L}_{X_{s,\varepsilon}} \right) ( x_s ) \right) U_{s}^{2} \right|^p}ds \notag \\
&=:C\int_0^t{\mathbb{E}\left| \frac{\tilde{X}_{s,\varepsilon}-U_s}{\varepsilon ^H}-V_s \right|}^pds+R_1+R_2.
	\end{align*}
By the assumption (A2), Cauchy-Schwarz inequality, \eqref{p7 eq1} and \eqref{p7 eq2}, we have
	\begin{equation}{\label{R2}}
R_1\le C\int_0^t{\mathbb{E}\left| \tilde{X}_{s,\varepsilon}^{2}-U_{s}^{2} \right|^p}ds
\le C\int_0^t{\sqrt{\mathbb{E}\left| \tilde{X}_{s,\varepsilon}^{}+U_{s}^{} \right|^{2p}\mathbb{E}\left| \tilde{X}_{s,\varepsilon}^{}-U_{s}^{} \right|^{2p}}}ds\le C\varepsilon ^{pH}.
	\end{equation}	
Also, by the Cauchy-Schwarz inequality
\begin{equation*}
	R_2\le C\int_0^t{\sqrt{\mathbb{E}\left| \left( \nabla ^2b\left( s,\cdot ,\mathcal{L}_{X_{s,\varepsilon}} \right) \left( x_s+\eta \left( X_{s,\varepsilon}-x_s \right) \right) -\nabla ^2b\left( s,\cdot ,\mathcal{L}_{X_{s,\varepsilon}} \right) ( x_s ) \right) \right|^{2p}\mathbb{E}\left| U_s \right|^{4p}}}ds
.
\end{equation*}
So that, by \eqref{p7 eq2} and the dominated
convergence theorem, we obtain that
\begin{equation}{\label{R3}}
	R_2 \rightarrow 0,\ \  as\ \  \varepsilon \rightarrow 0.
\end{equation}
Thus,
	\begin{equation}{\label{R123}}
		\mathbb{E}\left| \frac{\tilde{X}_{t,\varepsilon}-U_t}{\varepsilon^H}-V_t \right|^p \le C(R_1+R_2).
	\end{equation}
Combining \eqref{R2}, \eqref{R3} and \eqref{R123}, we can obtain \eqref{p9 eq1}.

\end{proof}

\begin{lemma}{\label{p10}}Suppose the assumptions (A1) and (A2) hold, then for each continuous and bounded function $g$ and each $t\in [0,t]$, we have
	\begin{equation}{\label{p10 eq1}}
		\lim_{\varepsilon \rightarrow 0}\frac{\mathbb{E}\left[ g\left( \tilde{X}_{t,\varepsilon} \right) \right] -\mathbb{E}\left[ g\left( U_t \right) \right]}{\varepsilon^H}=\frac{1}{||DU_t||_{L^2[0,T]}^2}\mathbb{E}[ g( U_t ) \delta ( V_tDU_t) ]
		,
	\end{equation}
	where $U_t$ and $V_t$ are defined as \eqref{U_t} and \eqref{V_t} respectively.
\end{lemma}
\begin{proof}
	According to the proof of Proposition 4.8 in \cite{ref15}, we can see that
	this proof does not use the exact form of $U_t$ and $V_t$. Thus, by combining
	Lemmas \ref{p8}--\ref{p9}, then repeating the deduction of Proposition 4.8 in \cite{ref15},
	we can easily get this result. Here we give a sketch of this proof.
	
	By using the formula (3.2) in \cite{formula 3.2},
	\begin{align*}
\mathbb{E}\left[ g\left( \tilde{X}_{t,\varepsilon} \right) \right] -\mathbb{E}\left[ g\left( U_t \right) \right] &=\mathbb{E}\left[ \int_{U_t}^{\tilde{X}_{t,\varepsilon}}{g\left( z \right) dz}\delta \left( \frac{DU_t}{||DU_t||_{L^2\left[ 0,T \right]}^{2}} \right) \right] \\
&\ \ \ \ -\mathbb{E}\left[ \frac{g\left( \tilde{X}_{t,\varepsilon} \right) \left< D\tilde{X}_{t,\varepsilon}-DU_t,DU_t \right> _{L^2\left[ 0,T \right]}}{||DU_t||_{L^2\left[ 0,T \right]}^{2}} \right] \\
&=\frac{1}{||DU_t||_{L^2\left[ 0,T \right]}^{2}}\mathbb{E}\left[ U_t\int_{U_t}^{\tilde{X}_{t,\varepsilon}}{g\left( z \right) dz} \right] \\
&\ \ \ \ -\frac{1}{||DU_t||_{L^2\left[ 0,T \right]}^{2}}\mathbb{E}\left[ g\left( \tilde{X}_{t,\varepsilon} \right) \left< D\tilde{X}_{t,\varepsilon}-DU_t,DU_t \right> _{L^2\left[ 0,T \right]} \right] .
	\end{align*}
    Then, by \eqref{dual 1}
	\begin{equation*}
		\delta \left( V_tDU_t \right) =V_tU_t+\left< DV_t,DU_t \right> _{L^2\left[ 0,T \right]}
.
	\end{equation*}
So that
\begin{align*}
&\frac{\mathbb{E}\left[ g\left( \tilde{X}_{t,\varepsilon} \right) \right] -\mathbb{E}\left[ g\left( U_t \right) \right]}{\varepsilon ^H}-\frac{1}{||DU_t||_{L^2\left[ 0,T \right]}^{2}}\mathbb{E}\left[ g\left( U_t \right) \delta \left( V_tDU_t \right) \right] \\
&=\frac{1}{||DU_t||_{L^2\left[ 0,T \right]}^{2}}\mathbb{E}\left[ \left( \frac{1}{\varepsilon ^H}\int_{U_t}^{\tilde{X}_{t,\varepsilon}}{g\left( z \right) dz}-g\left( U_t \right) V_t \right) U_t \right] \\
&\ \ -\frac{1}{||DU_t||_{L^2\left[ 0,T \right]}^{2}}\mathbb{E}\left[ \left( g\left( \tilde{X}_{t,\varepsilon} \right) -g\left( U_t \right) \right) \left< \frac{D\tilde{X}_{t,\varepsilon}-DU_t}{\varepsilon ^H},DU_t \right> _{L^2\left[ 0,T \right]} \right] \\
&\ \ -\frac{1}{||DU_t||_{L^2\left[ 0,T \right]}^{2}}\mathbb{E}\left[ g\left( U_t \right) \left< \frac{D\tilde{X}_{t,\varepsilon}-DU_t}{\varepsilon ^H}-DV_t,DU_t \right> _{L^2\left[ 0,T \right]} \right] \\
&=:\varLambda _{1,\varepsilon}-\varLambda _{2,\varepsilon}-\varLambda _{3,\varepsilon}.
\end{align*}
Thus, the proof of \eqref{p10 eq1} can be transformed into the proof of convergence for $\varLambda _{1,\varepsilon}$, $\varLambda _{2,\varepsilon}$ and $\varLambda _{3,\varepsilon}$.

Hence, by repeating the deduction of Proposition 4.8 in \cite{ref15}, we can see $\lim_{\varepsilon\to}\varLambda _{i,\varepsilon}=0$,  $i=1,2,3$, under the results
in Lemmas \ref{p8}--\ref{p9}.
\end{proof}

\textbf{Proof of Theorem} \ref{Th2}: Note that, by Theorem 3.9 in \cite{Fan23} combining \eqref{p7 eq1}, for
each $t \in (0,T]$, $U_t$ can be regarded as $Z_t$.
So that, we have $I(\tilde{X}_{t,\varepsilon}||Z_t)=I(\tilde{X}_{t,\varepsilon}||U_t)$.
Hence, by \eqref{FI vs TV} and Lemma \ref{p10}, we can get
\begin{equation*}
	\lim_{\varepsilon \rightarrow 0}\frac{1}{\varepsilon^H}\sqrt{I( \tilde{X}_{t,\varepsilon}||Z_t )}
	\ge \frac{1}{2||DU_t||_{L^2[0,T]}^2}|\mathbb{E}[g(U_t)\delta(V_t DU_t)]|=\frac{1}{2||DU_t||_{L^2[0,T]}^2}|\mathbb{E}[g(U_t)\mathbb{E}[\delta(V_tDU_t)|U_t]]|.
\end{equation*}
for any continuous function $g$ bounded by 1. Then,
by the routine approximation argument,
we can choose $g(x)=\text{sgn}\Big(\mathbb{E}[\delta(V_tDU_t)|U_t=x]\Big) $.
Hence we obtain (\ref{Th2 eq1}).

\bigskip



\textbf{Declaration of interests} ~The authors declare that they have no known competing financial interests or personal relationships that
could have appeared to influence the work reported in this paper.

	$\begin{array}{cc}
		\begin{minipage}[t]{1\textwidth}
			{\bf Tongxuan Liu}\\
			School of Mathematics, Nanjing University of Aeronautics and Astronautics, Nanjing
			211106, China\\
			\texttt{18061750781@163.com}
		\end{minipage}
		\hfill
	\end{array}$
	
	\bigskip
	$\begin{array}{cc}
		\begin{minipage}[t]{1\textwidth}
			{\bf Qian Yu}\\
			School of Mathematics, Nanjing University of Aeronautics and Astronautics, Nanjing
			211106, China\\
			\texttt{qyumath@163.com}
		\end{minipage}
		\hfill
	\end{array}$

\end{document}